\makeatletter
\def\ps@firstpage{\ps@plain
  \def\@oddfoot{\normalfont\scriptsize \hfil\thepage\hfil}
  \def\@oddhead{\article@logo\hss}}
\def\article@logo{\vbox to\headheight{\@parboxrestore \@logofont \noindent
\href{http://www.worldscientific.com/worldscibooks/10.1142/6719}{Differential Geometry and Its Applications}
\newline Proc. of the 10\textsuperscript{th} International Conference (Olomouc, August 27--31, 2007)
    \newline World Scientific Publishing, Hackensack, NJ, 2008, p.~635--642.
    \par\vss
  }}
\def\@logofont{\fontsize{8}{9.6\p@}\selectfont}
\makeatother
\documentclass{procs}
\usepackage[pdfstartview=FitH,pdfnewwindow=true,pdftitle={Variationality of geodesic circles},pdfdisplaydoctitle=true,colorlinks,pagebackref,bookmarksopen,breaklinks]{hyperref}
\newcommand{\bu}{{\boldsymbol u}}
\newcommand{\bw}{{\boldsymbol {u'}}}
\newcommand{\bp}{{\boldsymbol \pi}}
\newcommand{\bq}{{\boldsymbol \pi^{(1)}}}
\newcommand{\bpr}{{{\boldsymbol'}}}
\newcommand{\du}{{\boldsymbol {\dot u}}}
\newcommand{\pr}[2]{{\boldsymbol{#1\cdot#2}}}
\newcommand{\pw}[2]{({#1}\wedge{#2})}
\newcommand{\nbu}{{\left\|\boldsymbol{u}\right\|}}
\newcommand{\nw}[2]{{\left\|{#1}\wedge{#2}\right\|}}
\newcommand{\bcdot}{{\boldsymbol{\cdot}}}
\newcommand{\bpw}[2]{(\boldsymbol{#1}\wedge\boldsymbol{#2})}
\newcommand{\bnw}[2]{{\left\|\boldsymbol{#1}\wedge\boldsymbol{#2}\right\|}}
\begin{document}
\title{Variationality of geodesic circles in two dimensions}
\author{R. Ya. Matsyuk}
\address{Institute for Applied Problems in Mechanics and Mathematics,\\
15 Dudayev St., L'viv, 79005, Ukraine\\
\url{http://iapmm.lviv.ua/12/eng/files/st_files/matsyuk.htm}\\
E-mail: matsyuk@lms.lviv.ua, romko.b.m@gmail.com}
\begin{abstract}
This note treats the notion of Lagrange derivative for the third
order mechanics in the context of covariant Riemannian geometry.
The variational differential equation for geodesic circles in two
dimensions is obtained. The influence of the curvature tensor on
the Lagrange derivative leads to the emergence of the notion of
quasiclassical spin in the pseudo-Riemannian case.
\end{abstract}
\bodymatter
\thispagestyle{firstpage}
\section{Introduction}
This is a note on the variational formulation for the differential equations
of geodesic circles in two-dimensional Riemannian space, although the results
apply straightforward to the pseudo-Riemannian case. The
geodesic curves $x^i(t)$ obey with respect to the natural parameter $s$
the third order differential equation~\cite{matsyuk:YanoCon}
\[
\dfrac{D^3x^i}{ds^3}+g_{lj}\dfrac{D^2x^l}{ds^2}\dfrac{D^2x^j}{ds^2}\dfrac{Dx^i}{ds}=0\,,
\]
and they are exactly characterized by the property that the
(signed) Frenet curvature $k$ keeps constant along them. In view
of the Proposition~\ref{matsyuk:Ham} below we could have
immediately stated that the variational functional $\int\! kdt$
provides an answer to the problem, all the more that in two
dimensions $\sqrt{k^2}$ depends linearly on the second derivatives
of the coordinates along the curve thus producing exactly {\em the
third order\/} variational (called Euler-Poisson) equation.

However, we wish to investigate, to what extent this answer in
predefined by the limiting case of the Euclidean space --- the
local model of the Riemannian one. With this idea in mind we start
by recalling one solution\cite{matsyuk:MatMet} of the {\em invariant inverse
variational problem\/} in two-dimensional Euclidean space for
a third order variational equation possessing the first
integral~$k$.

Before proceeding further, it is necessary to agree
about some notations and to recall some basic calculus on the
second order Ehresmann velocity space
$T^2M\overset{\mathrm{def}}=J^2_0(\mathbb R,M)$ of jets from
$\mathbb R$ to our manifold $M$ starting at $0\in \mathbb R$ (as
possible source of references we can recommend, for example,
Refs.~\refcite{matsyuk:Ibort,matsyuk:deLeon,matsyuk:Yano}).

\section{Calculus on the higher order velocities space}

Let $u^i,\dot u^i$ denote the standard fiber coordinates in $T^2M$. In case
of an affine space $M$, we use the vector notations $\bu,\du$ for that
tuple. In future we shall profoundly also use another tuple of coordinates,
namely, $\bu,\bw$, where
\begin{equation}\label{matsyuk:uprime}
u\boldsymbol'^i=\dot u^i+\Gamma^i_{lj}u^lu^j
\end{equation}
stands for the
covariant derivative of~$\bu$.
Let us recall some operators acting in the algebra of differential forms, defined
on the velocity spaces of the sequential orders:
\begin{list}{$\bullet$}{\setlength{\leftmargin}{10pt}}
\item The total derivative:
\[d_Tf=u^i\dfrac{\partial f}{\partial x^i}+\dot u^i\dfrac{\partial f}{\partial u^i}
    +\ddot u^i\dfrac{\partial f}{\partial \dot u^i}+\ldots
\]
This is a derivation of degree zero and of the type $d$, {\it i.e.} who
commutes with the exterior differential: $dd_T=d_Td$.
\item
For each $k=1,2,3$, let $u_{(k)}^i=\overset{\overbrace{\dotsm}^{k}}{u}{}^i$,
$u_{(0)}^i=u^i$, $x_{(k)}^i=u_{(k-1)}^i$, and $x_{(0)}^i=x^i$.
For each $r=0,1,2,3,\ldots$, we recall the following derivations of degree zero and of the type~$i$,
{\it i.e.} who produce zeros while acting on the ring of functions:
\begin{equation}\label{matsyuk:iota}
\begin{split}
& \iota_r(f)=0\,,\\
\iota_r(dx_{(k)}^i)
=\frac{k!}{(k-r)!}dx_{(k-r)}^i,
& \quad \mathrm{and}\quad \iota_r(dx_{(k)}^i)=0,\quad \mathrm{if}\quad r>k.
\end{split}
\end{equation}
\item The Lagrange derivative $\delta$:
\begin{equation}\label{matsyuk:delta}
\delta=(\iota_0 - d_T\iota_1 + \frac 1{2!}d_T^2\iota_2 - \frac 1{3!}d_T^3\iota_3+\ldots)\,d\,,
\end{equation}
that satisfies $\delta^2=0$.
\end{list}
Let some system of the third order ordinary differential equations
$\mathcal E_i(x^j,u^j,\dot u^j, \ddot u^j)$ be put in the shape of
a covariant object $\epsilon$:
\[
\epsilon=\mathcal E_i(x^j,u^j,\dot u^j, \ddot u^j)dx^i\,.
\]
The variationality criterion reads:
If $\delta \epsilon=0$, then the system $\mathcal E_i$ is variational,
{\it i.e.} locally there exists some function $L$, such that $\epsilon=\delta L$.

The right action of the prolonged group
$GL_{(2)}(\mathbb R)\overset{\mathrm{def}}=\overset{\:\circ}J{}^2_0(\mathbb R,\mathbb R)_0$ of
parameter transformations (invertible transformations of the independent variable $t$)
on $T^2M$
gives rise to the so-called fundamental fields on $T^2M$:
\[
\zeta_1=u^i\dfrac{\partial}{\partial u^i}+2\dot u^i\dfrac{\partial}{\partial\dot
u^i},\quad \zeta_2=u^i\dfrac{\partial}{\partial\dot u^i}\,.
\]

A function $f$ defined on $T^2M$ does not depend on the change of independent
variable $t$ (so--called parameter--independence) if and only if
\begin{equation}\label{matsyuk:parind}
\zeta_1f=0,\quad\zeta_2f=0\,.
\end{equation}

On the other hand, a function $L$ defined on $T^2M$ constitutes a parameter--independent variational
problem with the functional
$
\int L(x^j,u^j,\dot u^j)dt
$
 if and only if the following Zermelo conditions are satisfied:
\begin{equation}\label{matsyuk:Zermelo}
\zeta_1L=L,\quad\zeta_2L=0\,.
\end{equation}

Let us introduce the generalized momenta:
\[p^{(1)}_i=\dfrac{\partial L}{\partial\dot u^i},\quad
p_i=\dfrac{\partial L}{\partial u^i}-d_Tp^{(1)}_i\,.
\]
These satisfy the relation:
\begin{equation}\label{matsyuk:momenta relation}
p^{(1)}{}_idu^i+p_idx^i=\iota_1dL-\dfrac{1}{2}d_T\iota_2dL\,.
\end{equation}

The Euler--Poisson equation is given by $\delta L=0$, or, equivalently, by
\[
\dot p_idx^i=\frac{\partial L}{\partial x^i}dx^i\,.
\]

\enlargethispage*{20pt}
The Hamilton function is given by:
\[
H=p^{(1)}_i\dot u^i+p_iu^i-L\,.\]
\begin{lemma}\label{matsyuk:ham}
\[
H=\zeta_1L-d_T\zeta_2L-L\,.
\]
\end{lemma}
\begin{proposition}\label{matsyuk:Ham}
If a function $L_{\mathrm{II}}$ is parameter--independent and a function $L_{\mathrm
I}$ constitutes a parameter-independent variational problem, then
 $L_{\mathrm{II}}$ is constant along the extremals of $L=L_{\mathrm{II}}+L_{\mathrm{I}}$.
\end{proposition}
\begin{proof}
By Lemma~\ref{matsyuk:ham} and in course of the properties~(\ref{matsyuk:parind})
and~(\ref{matsyuk:Zermelo}) we calculate $H_{L_{\mathrm{II}}+L_{\mathrm{I}}}=\zeta_1(L_{\mathrm{II}}+L_{\mathrm{I}})
 -d_T\zeta_2(L_{\mathrm{II}}+L_{\mathrm{I}})-L=-L_{\mathrm{II}}$.
But as far as the Hamilton function is constant of motion, so is the~$L_{\mathrm{II}}$.
\end{proof}

\section{The Lagrange derivative in Riemannian space}
In Riemannian space with symmetric connection the covariant
differential of a vector field $\boldsymbol\xi$ is a vector field
valued semibasic differential form~$D\boldsymbol\xi$ calculated
according to the formula
\begin{equation}\label{matsyuk:CovDiff}
(D\boldsymbol\xi)^i=d\xi^i+\Gamma^i_{lj}\xi^jdx^l.
\end{equation}
The fundamental application of the curvature tensor, from which this
note profits, provides the commutator of the subsequent derivations, the one
that replaces the known Schwarz lemma:
\begin{equation}\label{matsyuk:Duprime}
(D\bu){\boldsymbol'}{}^i=(D(\bw))^i+R_{ljq}{}^iu^ju^qdx^l\,.
\end{equation}

We also recall that, on the other hand, the first order derivations commute:
\begin{equation}\label{matsyuk:commute}
(dx)\boldsymbol'=D\bu\,.
\end{equation}

Given some local coordinate expression of a function,
\[
L(x^i, u^i, \dot u^i)\,,
\]
we wish to introduce generalized momenta $\pi_i$ and $\pi^{(1)}{}_i$, calculated
with respect to the alternative set of coordinates in $T^2M$, namely,
$x^i$,$u^i$, $u\boldsymbol'^i$, where the transition functions are presented
by~(\ref{matsyuk:uprime}).
\begin{definition}\label{matsyuk:Def 3.1}
Let
\[
\pi^{(1)}{}_i=\dfrac{\partial L}{\partial u\boldsymbol'^i}\,,\quad \pi_i=\dfrac{\partial L}{\partial u^i}-\pi^{(1)}\boldsymbol'{}_i\,.
\]
\end{definition}
\begin{proposition}\label{matsyuk:Prop 3.1}
In Riemannian space the generalized momenta satisfy the relation, analogous
to~(\ref{matsyuk:momenta relation}):
\[
\bq D\bu+\bp dx=\iota_1dL-\dfrac{1}{2}(\iota_2dL)\bpr
\]
\end{proposition}
\begin{proof}
First let us calculate by the reason of formulas~(\ref{matsyuk:iota})
and~(\ref{matsyuk:uprime}):
\[
\iota_1 d\bu=dx,\quad \iota_1d\bw=2D\bu,\quad  \iota_2d\bw=2dx\,.
\]
For the differential of Lagrange function,
\begin{equation}\label{matsyuk:dL}
dL=\frac{\partial L}{\partial x}
+\frac{\partial L}{\partial \bu}d\bu+\frac{\partial L}{\partial \bw}d\bw\,,
\end{equation}
we then check:
\[
\iota_2dL=2\,\frac{\partial L}{\partial \bw} dx\,.
\]

Now calculate:
\begin{align*}
\iota_1dL & =\frac{\partial L}{\partial \bu}dx+
2\,\frac{\partial L}{\partial \bw}D\bu \\
&=\frac{\partial L}{\partial \bu}dx+
2\,\left(\frac{\partial L}{\partial \bw}dx\right)'
-2\,\left(\frac{\partial L}{\partial \bw}\right)'dx \\
&=
\frac{\partial L}{\partial \bu}dx+ \left(\iota_2dL\right)'
-2\,\left(\frac{\partial L}{\partial \bw}\right)' dx\,,
\end{align*}
from where and from the Definition~\ref{matsyuk:Def 3.1} it follows immediately that
\begin{align*}
\bp dx&=\iota_1dL-(\iota_2dL)\bpr+
\left(\frac{\partial L}{\partial \bw}\right)' dx\\
&=
\iota_1dL-\dfrac{1}{2}(\iota_2dL)\bpr
-\dfrac{1}{2}(\iota_2dL)\bpr+(\bq)\bpr dx \\
&=
\iota_1dL-\dfrac{1}{2}(\iota_2dL)\bpr
-(\bq dx)\bpr+(\bq)\bpr dx \\
&=
\iota_1dL-\dfrac{1}{2}(\iota_2dL)\bpr - \bq D\bu
\end{align*}
by virtue of~(\ref{matsyuk:commute}).
\end{proof}
\begin{proposition}
In Riemannian space the Euler--Poisson equation for a second order Lagrange
function reads:
\begin{equation}\label{matsyuk:E-P}
\bp\bpr dx+\pi^{(1)}{}_iR_{ljq}{}^iu^ju^qdx^l
=\frac{\partial L}{\partial x^l}dx^l
-\frac{\partial L}{\partial u^i}\Gamma^i_{lj}u^jdx^l
-\frac{\partial L}{\partial u'^i}\Gamma^i_{lj}u'^jdx^l
\end{equation}
\end{proposition}
\begin{proof}
From~(\ref{matsyuk:delta}) and from Proposition~\ref{matsyuk:Prop 3.1}
we obtain
\[\delta L=\iota_0dL-d_T(\bp dx+\bq D\bu).
\]
While the expression
in the parenthesis constitutes a geometrical invariant, it is possible to
replace $d_T$ by the covariant derivative, after what
by direct calculation  we obtain in virtue of (\ref{matsyuk:commute}) and
of~(\ref{matsyuk:Duprime}):
\begin{multline*}
\delta L=\iota_0dL-(\bp dx + \bq D\bu)\bpr
=\iota_0dL-\bp\bpr dx
-\big(\bp  + \bq\bpr \big) D\bu
- \bq (D\bu)\bpr \\
=\iota_0dL-\bp\bpr dx -\frac{\partial L}{\partial \bu}D\bu
-\pi^{(1)}{}_i \left(D(\bw)^i+R_{ljq}{}^iu^ju^qdx^l\right),
\end{multline*}
 and the proof ends by substituting~(\ref{matsyuk:dL}) into $i_0dL\equiv dL$ here
 and by applying~(\ref{matsyuk:CovDiff}).
\end{proof}

\section{The two-dimensional variational concircular geometry}
As promised, we first cite one result, concerning the invariant inverse
variational problem in two dimensional Euclidean space~\cite{matsyuk:MatMet}.
\begin{proposition}
Let some system of third order differential equations
\begin{equation}\label{matsyuk:epsilon}
 \mathcal E_i(x^j,u^j,\dot u^j, \ddot u^j)=0
\end{equation}
satisfy the conditions:
\begin{romanlist}
\item
$
\delta \,\mathcal E_idx^i\,=\,0
$
\item
The system~(\ref{matsyuk:epsilon}) possesses Euclidean symmetry
\item
The Euclidean geodesics $\du=\boldsymbol0$ enter in the set of solutions
of~(\ref{matsyuk:epsilon})
\item
$d_Tk=0$ along the solutions of~(\ref{matsyuk:epsilon})
\end{romanlist}
Then
\[
\mathcal E_i
=\frac{e_{ij}\ddot u^j}{\|\bu\|^3}-3\,\frac{(\pr{\du}{\bu})}{\|\bu\|^5}e_{ij}\dot u^j
+ m\,\frac{\|\bu\|^2\dot u_i-(\pr{\du}{\bu})u_i}{\|\bu\|^3}\;.
\]
\end{proposition}
This system may be obtained from the Lagrange function
\begin{equation}\label{matsyuk:flat}
L=\dfrac{e_{ij} u^i \dot u^j}{\|\bu\|^3}-m\,\|\bu\|\,.
\end{equation}
The first addend in~(\ref{matsyuk:flat}) is sometimes called {\em
the signed Frene curvature\/} in $\mathbb E^2$. This, along
with the observation that in two dimensional Riemannian space the Frenet curvature
\begin{equation}\label{matsyuk:Frene}
k=\dfrac{\nw{\bu}{\bw}}{\nbu^3}=\pm\dfrac{*\,\pw{\bu}{\bw}}{\nbu^3}
\end{equation}
 depends linearly on $\bw$ and
thus produces at most third order Euler--Poisson equation,
suggests the next assertion, based
on Proposition~\ref{matsyuk:Ham}:
\begin{proposition}
The variational functional $\int\, (k - m\,\|\bu\|)\,dt$ produces
geodesic circles in two dimensional Riemannian space.
\end{proposition}

It remains to calculate the Euler--Poisson expression for the
Lagrange function~(\ref{matsyuk:Frene}).
In the process of calculations it is convenient to profit from the
 exeptional
properties of vector operations in two dimensions. Namely, the following two
relations for arbitrary vectors hold:
\begin{gather*}
\bpw a b \bcdot \bpw v w=\pm\,\bnw a b  \bnw v w \,,\\
\intertext{and}
\bnw a b (\pr b c)+\bnw b c (\pr a b) = \bnw a c (\pr b b)\,.
\end{gather*}
The above simplifications bring much release to otherwise very
laborious calculations.

We start with the
momentum $\bq$:
\begin{align*}
&\pm\bq dx=-\dfrac{\pw{dx}{\bu}\bcdot\pw{\bu}{\bw}}{\nbu^3\nw{\bu}{\bw}}
=-\dfrac{\nw{dx}{\bu}}{\nbu^3};\\
&\pm\bq\bpr dx= -\dfrac{\nw{dx}{\bw}}{\nbu^3}
+3\,\dfrac{\nw{dx}{\bu}}{\nbu^5}(\bu\bcdot\bw).
\end{align*}

Based on Definition~\ref{matsyuk:Def 3.1} we now calculate~$\bp$:
\[
\pm\bp dx=2\,\dfrac{\nw{dx}{\bw}}{\nbu^3}-3\,\dfrac{\nw{dx}{\bu}(\bu\bcdot\bw)}{\nbu^5}
-3\,\dfrac{(dx\bcdot\bu)\nw{\bu}{\bw}}{\nbu^5}
=-\dfrac{\nw{dx}{\bw}}{\nbu^3}
\]
In terms of the Hodge star operator the derivative of the momentum~$\bp$
may be put in the form
\[
\bp\bpr=\dfrac{*\bu\bpr\bpr}{\nbu^3}-3\,\dfrac{*\bu\bpr}{\nbu^5}(\bu\bcdot\bw)\,,
\]
which agrees with the flat Euclidean case.

For the Lagrange function~(\ref{matsyuk:Frene}) it is easy to verify that
\[
\dfrac{\partial k}{\partial x^l}dx^l-\dfrac{\partial k}{\partial u^i}\Gamma^i_{lj}u^jdx^l
-\dfrac{\partial k}{\partial u'^i}\Gamma^i_{lj}u'jdx^l=0.
\]
The proof consists in direct calculations and founds on the skew-symmetric
property of the Christoffel symbols in Riemannian geometry:
\[
g_{jl}\Gamma^l_{qi}+g_{il}\Gamma^l_{qj}=\frac{\partial g_{ij}}{\partial x^q}.
\]

Going back to the Euler--Poisson equation~(\ref{matsyuk:E-P}) it is now
facile to obtain the variational equation
for the full  Lagrange function $L=k-m\nbu$:
\begin{equation}\label{matsyuk:vareqcircle}
-\dfrac{*\bu\bpr\bpr}{\nbu^3}+3\,\dfrac{*\bu\bpr}{\nbu^5}(\bu\bcdot\bw)\
+ m\,\frac{\|\bu\|^2\bw-(\pr{\bw}{\bu})\bu}{\|\bu\|^3}
=\pi^{(1)}{}_iR_{ljq}{}^iu^ju^q.
\end{equation}

The term on the right in pseudo-Riemannian case physically may be interpreted
as a spin force\cite{matsyuk:Mathisson} if, following Ref.~\refcite{matsyuk:Leiko},
we formally introduce spin tensor as $S=\pw{\bu}{\bw}$.

In fact, one checks that in terms of the tensor~$S$ the right hand side of
equation~(\ref{matsyuk:vareqcircle}) may be rewritten as
$\dfrac{R_{ljqi}{}u^jS^{qi}}{\nbu\nw{\bu}{\bw}}$.

\section{Acknowledgments}
This work was supported by the Grant GA\v CR 201/06/0922 of Czech Science Foundation.

\end{document}